\documentclass[11pt]{amsart}
\usepackage{graphicx}
\usepackage{graphicx}
\usepackage{amsfonts, amsmath, amsthm, amssymb}
\usepackage{cite}
\usepackage{hyperref}
\newtheorem{thm}{Theorem}[section]
\newtheorem{cor}[thm]{Corollary}
\newtheorem{lem}[thm]{Lemma}
\newtheorem{prop}[thm]{Proposition}
\newtheorem{conj}[thm]{Conjecture}
\newtheorem{exm}[thm]{Example}
\newtheorem{defn}[thm]{Definition}
\newtheorem{rem}[thm]{Remark}
\numberwithin{equation}{section}

\newcommand{\norm}[1]{\left\Vert#1\right\Vert}
\newcommand{\abs}[1]{\left\vert#1\right\vert}

\newcommand{\Bc}{\scalebox{1.8}{$\chi$}}% big chi

\def\eb{{\mathbf{e}}}

\def\gf{{\frak F}}

\def\bh{{\mathbb H}}

\def\bn{{\mathbb N}}
\def\bp{{\mathbb P}}

\def\br{{\mathbb R}}

\def\a{\alpha}

\def\l{\lambda} 

\def\m{\mu}

\def\s{\sigma} 
\def\t{\tau}

\def\w{\omega} \def\O{\Omega}

\def\pb{{\mathbf{p}}}
\def\xb{{\mathbf{x}}}
\def\yb{{\mathbf{y}}}
\def\zb{{\mathbf{z}}}

\def\yb{{\mathbf{y}}}
\def\a{\alpha}
\begin{document}
    \title[Uniqueness of fixed points]{Uniqueness of fixed points of $ b$-bistochastic quadratic stochastic operators and
 associated nonhomogenous Markov chains}

   \author{Farrukh Mukhamedov}
\address{Farrukh Mukhamedov\\
 Department of Computational \& Theoretical Sciences\\
Faculty of Science, International Islamic University Malaysia\\
P.O. Box, 141, 25710, Kuantan\\
Pahang, Malaysia} \email{{\tt far75m@yandex.ru}, {\tt
farrukh\_m@iium.edu.my}}

\author{Ahmad Fadillah Embong}
\address{Ahmad Fadillah\\
 Department of Computational \& Theoretical Sciences\\
Faculty of Science, International Islamic University Malaysia\\
P.O. Box, 141, 25710, Kuantan\\
Pahang, Malaysia} \email{{\tt ahmadfadillah.90@gmail.com}}

    \begin{abstract}
In the present paper, we consider a class of quadratic stochastic
operators (q.s.o.) called $ b- $bistochastic q.s.o. We include
several properties of $ b- $bistochastic q.s.o. and their
dynamical behavior. One of the main findings in this paper is the
description on the uniqueness of the fixed points. Besides, we list
the  conditions on strict contractive $ b- $bistochastic q.s.o. on
low dimensional simplices and it turns out that, the uniqueness of
the fixed point does not imply strict contraction. Finally, we
associated Markov measures with \textit{b-}bistochastic q.s.o. On
a class of \textit{b-}bistochastic q.s.o. on finite dimensional simplex, the defined measures
were proven to satisfy the mixing property. Moreover, we show that
Markov measures associated with a class of $ b-
$bistochastic q.s.o  on one dimensional simplex meets the
absolute continuity property.

\vskip 0.3cm \noindent {\it Mathematics Subject
           Classification}: 46L35, 46L55, 46A37.\\
        {\it Key words}: quadratic stochastic operators;  $b-$order;
        $b-$bistochastic, mixing, unique fixed point, absolute continuity.
    \end{abstract}

\maketitle

\section{Introduction}

The history of quadratic stochastic operators (q.s.o.) can be traced
back to Bernstein's work \cite{1} where such kind of operators
appeared from the problems of population genetics (see also
\cite{11}). Such kind of operators describe time evolution of
variety species in biology are represented by so-called
Lotka-Volterra(LV) systems \cite{29}.
 Nowadays, scientists are interested
    in these operators, since they have a lot of applications
    especially  in modelings
    in many different fields such as biology \cite{8,20},
    physics \cite{21,25}, economics and mathematics
    \cite{11,K,20,25}.

Let us recall how q.s.o. appears in biology \cite{11}. The time
evolution of species in biology can be comprehended by
    the following situation. Let $I = \{1,2,\dots,n\}$ be the $n$ type
    of species (or traits) in a population and we denote $x^{(0)} =
    (x_1^{(0)},\dots,x_n^{(0)})$ to be the probability distribution of
    the species in an early state of that population.
    By $P_{ij,k}$ we mean  the
    probability of an individual in the $i^{th}$ species and $j^{th}$
    species to cross-fertilize and produce an individual from $k^{th}$
    species (trait). Given $x^{(0)} =
    (x_1^{(0)},\dots,x_n^{(0)})$, we can find the probability
    distribution of the first generation, $x^{(1)} =
    (x_1^{(1)},\dots,x_n^{(1)})$ by using a total probability, i.e.
    \begin{eqnarray*}
        x_k^{(1)}=\sum\limits_{i,j=1}^n P_{ij,k}x_i^{(0)}x_j^{(0)}, \ \
        k\in\{1,\dots,n\}. \label{operator}
    \end{eqnarray*}

This relation defines an operator which is denoted by $V$ and it
is called \textit{quadratic stochastic operator (q.s.o.)}. In
other words, each q.s.o. describes the sequence of generations
    in terms of probabilities distribution if the values of $P_{ij,k}$
    and the distribution of the current generation are given. In
    \cite{6,MG2015}, it has given along self-contained exposition of the recent
    achievements and open problems in the theory of the q.s.o. The
    main problem in the nonlinear operator theory is to study the
    behavior of nonlinear operators. Presently, there are only a small
    number of studies on dynamical phenomena on higher dimensional
    systems, even though they are very
    important. In case of q.s.o., the difficulty of the problem depends
    on the given cubic matrix $(P_{ijk})^m_{i,j,k=1}$.

In \cite{greed} a
    new majorization was introduced, and it opened a path for the study to
    generalize the theory of majorization by Hardy, Littlewood and
    Polya \cite{inqlty}. The new majorization has an advantage as compared
    to the classical one, since it can be defined as a partial order on
    sequences. While the classical one is not an antisymmetric
    order (because any sequence is majorized by any of its
    permutations), it is only defined as a preorder on
    sequence \cite{greed}.
Most of the works in the mentioned paper were devoted to the
investigation of majorized linear operators (see
\cite{inqlty,greed}). Therefore, it is natural to study nonlinear
majorized operators.

In what follows, to differentiate between the terms majorization
and classical majorization that was popularized by Hardy et
al.\cite{inqlty}, we call majorization as $b-$order (which is
denoted as $\leq^{b}$) while classical majorization as
majorization (which is denoted as $\prec $) only. In
\cite{def_qbo} it was introduced and studied q.s.o. with a
        property $V(\xb) \prec \xb$ for all $\xb
        \in S^{n-1}$. Such an operator is called \textit{bistochastic}.
In \cite{ME2015}, it was proposed to a definition of bistochastic
q.s.o. in terms of $ b$- order. Namely, a q.s.o. is called
\textit{$ b- $bistochastic} if $V(\xb) \leq^{b} \xb$ for all $\xb
S^{n-1}$.

In this paper we continue our previous investigations on
$b$-bistochastic operators. Namely, in Section 2 we recall several
properties of $b$-bistochastic operators which will be used in the
next sections. In Section 3, we describe sufficient conditions for
the uniqueness of fixed points of $b$-bistochastic operators. Then
in Section 4, we establish that $b$-bistochastic q.s.o. with
unique fixed point may not be a strict contraction. In Section 5,
we construct non-homogeneous Markov measures $\m_{V,\xb}$
associated with $b$-bistochastic q.s.o. $V$ and initial state
$\xb\in S^{n-1}.$ We prove that the Markov measure $\m_{V,\xb}$,
associated with $b$-bistochastic q.s.o. having unique fixed point,
satisfied the mixing property. Note that this kind of construction
of a Markov measures associated with q.s.o. was first considered
in \cite{GanistogenQO,ganisarymsakovcenlimthmqc}. Certain
properties of the associated Markov chains have been investigated
in several paper such as
\cite{pulkaMixnErPrNonMC,farL1ErNonDiscMP,farSupAkmaMarPQSP}
According to the construction, the Markov measures  depend on the
initial state of $\xb$. Given q.s.o. it is interesting to know how
the measures $\m_{V,\xb}$ relate to each other for different
initial states. In final Section 6, we examine the absolute
continuity of these measures for $b$-bistochastic q.s.o. We notice
that similar kind of study was done in \cite{ganiMeaCorQO} for
Mendelian q.s.o. We stress that, in that situation, the
non-homogeneous Markov chain reduced to the Bernoulli measures.
But in our situation, the measure $ \m_{\xb,V} $ are purely
non-homogeneous Markov. Therefore, the absolute continuity in this
situation is not evident.

\section{$b$-bistochastic operators} \noindent

In this section we recall necessary definitions and facts about
$b$-bistochastic operators.

Throughout this paper we consider the simplex
    \begin{eqnarray}\label{kthresidue}
    S^{n-1} = \left\{\textbf{x}=(x_1,x_2,...,x_n)\in \mathbb{R}^n\quad
    |x_i\geq0,\quad \sum\limits_{i=1}^{n}x_i = 1\right\}
    \label{eqn1.1}.
    \end{eqnarray}
    Moreover, by $ ri \ S^{n-1} $, we mean the relative interior of $S^{n-1}$, i.e.
    \[ ri \ S^{n-1} = \left\{ x \in S^{n-1} \vert x_{1} \cdot x_{2} \cdots x_{n} \neq 0 \right\}. \]

    Define functionals  $\mathcal{U}_k : \mathbb{R}^{n} \rightarrow
    \mathbb{R}$ by:
    \begin{eqnarray}
    \mathcal{U}_k (x_1,\dots, x_n) = \sum_{i=1}^{k}x_i \ \
    \textrm{where}\ \ k=\overline{1,n-1}. \label{lyapunov_fn}
    \end{eqnarray}

    Let $\xb,\yb\in \ S^{n-1}$. We say that $ \xb $ is
    \textit{b-ordered} or \textit{b-majorized}
    by $ \yb $ ($\xb \leq^{b} \yb$) if and only if $\mathcal{U}_k(\xb)
    \leq \mathcal{U}_k(\yb),\ \ \textrm{for all} \  \ k \in \{1,\dots, n-1\}.$

The introduced relation is partial order i.e. it satisfies the following conditions: %for any $\xb,\yb,\zb \in S^{n-1}$:
    \begin{itemize}
        \item[(i)] For any $ \xb \in S^{n-1} \textmd{ one has } \xb \leq^b \xb, \ \ $
        \item[(ii)] If $ \xb \leq^b \yb \textmd{ and } \yb \leq^b \xb   \textmd{ then } \xb = \yb,$
        \item[(iii)] If $ \xb\leq^b \yb, \textmd{ and } \yb \leq^b \zb  \textmd{ then } \xb \leq^b \zb.$
    \end{itemize}

Using the defined order, one can define the classical majorization
    \cite{Mar}. First, recall that for any $\xb= (x_1, x_2, \dots, x_n)
    \in S^{n-1}$, by $\xb_{[\downarrow]} = (x_{[1]}, x_{[2]},
    \dots, x_{[n]})$ one denotes the \textit{rearrangement} of $\xb$ in non-increasing order,
    i.e. $x_{[1]} \geq x_{[2]} \geq \dots \geq x_{[n]}$. Take $\xb, \yb\in S^{n-1}$,
    then it is said that an element \textit{$\xb$ is majorized by $\yb$} and denoted $\xb \prec \yb$ if
    $\xb_{[\downarrow]}\leq^b\yb_{[\downarrow]}$. We refer the reader to
    \cite{Mar} for more information regarding to this topic.

Any operators $V$ with $V(S^{n-1})\subset S^{n-1}$ is called
    \textit{stochastic}.
\begin{defn}
    A stochastic operator $V$ is called \textit{$b-$bistochastic} if one has $V(\xb)\leq^b\xb$ for all $\xb\in
    S^{n-1}$.
\end{defn}
    Note that, the simplest nonlinear operators are quadratic ones. Therefore, we restrict ourselves to such kind of
    operators. Namely, a stochastic operator $ V:S^{n-1} \rightarrow S^{n-1}$ is called
    \textit{quadratic stochastic operator (q.s.o.)} if $V$ has the
    following form:
    \begin{eqnarray}\label{qso}
    V(\xb)_k= \sum\limits_{i,j=1}^{n}P_{ij,k}x_ix_j,\quad  k =
    1,2,\dots,n, \ \  \xb =(x_1,x_2,\dots,x_n)\in S^{n-1},
    \end{eqnarray}
    where $\{P_{ij,k}\}$ are the heredity coefficients with the following properties:
    \begin{eqnarray}\label{eqn_propty_HC}
    P_{ij,k}\geq0, \quad P_{ij,k}=P_{ji,k},  \quad
    \sum\limits_{k=1}^{n}P_{ij,k} = 1, \quad  i,j,k=1,2,\dots,n \ .
    \end{eqnarray}

    \begin{rem}
If a q.s.o. $V$ satisfies $V(\xb) \prec \xb$ for all $\xb
        \in S^{n-1}$, then it is called \textit{bistochastic} \cite{def_qbo}. In
        our definition, we are taking $ b- $order instead of the majorization.
            \end{rem}

        Let  $F : \mathbb{R}^n \rightarrow \mathbb{R}^n$ be a given
    mapping by
    \[ F(\xb) = (f_{1}(\xb),f_{2}(\xb),...,f_{n}(\xb)), \ \xb\in\br^n \ . \]
    If $F$ is differentiable, then the Jacobian of $F$ at a point
    $\xb$ is defined by
    $$
    J(F(\xb))=
    \begin{bmatrix}
    \frac{\partial f_{1}}{\partial x_{1}}  & \ldots & \frac{\partial f_{1}}{\partial x_{n}} \\
    \vdots & \vdots & \vdots \\
    \frac{\partial f_{n}}{\partial x_{1}} & \ldots & \frac{\partial f_{n}}{\partial x_{n}} \\
    \end{bmatrix}\ .
    $$

    A point $\xb_0$ is called \textit{fixed point} of $F$ if one has
    $F(\xb_0)=\xb_0$.

    \begin{defn}
        A fixed point $\xb_0$ is called \textit{hyperbolic} if the
        absolute value of every eigenvalue $\l$ of the Jacobian at $\xb_{0}$ satisfies  $|\l|\neq 1$. Let $\xb_{0}$ be
        a hyperbolic fixed point, then
        \begin{itemize}
            \item[1.] $\xb_{0}$ is called \textit{attractive} if every
            eigenvalue of $J(F(\xb_0))$ satisfies $|\l|<1$.

            \item[2.] $\xb_{0}$ is called \textit{repelling} if every
            eigenvalue of $J(F(\xb_0))$ satisfies $|\l|>1$
        \end{itemize}
        %If $\xb_{0}$ is not hyperbolic, then it is called
        %\textit{neutral}.
        \label{def_Jacob}
    \end{defn}

    %Let us define
    %\begin{itemize}
    %\item[(i)] $\overline{\xb}_k =
    %(\underbrace{0,0,\dots,0}_{k},\dfrac{1}{n-k}, \dfrac{1}{n-k},
    %\dots , \dfrac{1}{n-k})$, where  $k=\{0,\dots,n-1\},$
    % \item[(ii)] $\eb_m = (\underbrace{0,0,
    %\dots, 0, 1}_m, 0,\dots, 0) \ $, \ \ $m\in\{1,\dots,n\}.$
    %\end{itemize}

  In \cite{ME2015} we have proved the following.

    \begin{thm}\cite{ME2015} \label{thm_des_b-bstc_2D}
        Let $V$ be a $b-$bistochastic q.s.o. defined
        on $S^{n-1}$, then the following statements hold:
        \begin{itemize}
            \item[(i)] $\sum\limits_{m=1}^{k}\sum\limits_{i,j=1}^{n}P_{ij,m}
            \leq kn, \ k \in \{1,\dots,n \}$
            \item[(ii)]$P_{ij,k} =0$ for all $i,j\in\{k+1,\dots,n\}$
            where $k\in\{1,\dots,n-1\}$
            \item[(iii)]$P_{nn,n}=1$
            \item[(iv)] for every $\xb\in S^{n-1}$ one has
            $$
            \begin{array}{ll}
            V(\xb)_k = \sum\limits_{l=1}^{k}P_{ll,k}x^2_l + 2
            \sum\limits_{l=1}^{k}\sum\limits_{j=l+1}^{n}P_{lj,k}x_lx_j  \ \
            where \ \ k = \overline{1,n-1}\\[4mm]
            V(\xb)_n = x_n^2 +  \sum\limits_{l=1}^{n-1}P_{ll,n}x^2_l + 2
            \sum\limits_{l=1}^{n-1}\sum\limits_{j=l+1}^{n}P_{lj,n}x_lx_j.  \
            \end{array}
            $$
            \item[(v)] $P_{lj,l} \leq \dfrac{1}{2}$ for all $j\geq l+1$,
            $l\in\{1,\dots,n-1\}$.
            % \item[(vi)]$P_{ll,l} + 2
            %\sum\limits_{j=l+1}^{n-1}P_{lj,l} +2P_{ln,l}(l-n) \leq 1$, where
            %$l\in\{1,\dots,n-1\}$.
        \end{itemize}
        \label{Prop_Pprties_V}
    \end{thm}

 \begin{thm}\cite{ME2015}\label{thm_lim_pt_e}
        Let $V$ be a $ b- $bistochastic q.s.o. defined on $S^{n-1}$, then for every
        $\xb\in S^{n-1}$ the limit $\lim\limits_{m \rightarrow
        \infty}V^{(m)}(\xb)=\overline{\xb}$ exists. Moreover,
        $\overline{\xb}$ is a fixed point of $V$.
    \end{thm}

 Denote $\bp = (\bp_{1},\dots, \bp_{n}):=(0,\dots,0,1)$.

    \begin{prop}
        Let $V:S^{n-1} \rightarrow S^{n-1}$ be a $b-$bistochastic q.s.o., then $\bp$ is its
        fixed point.
    \end{prop}

    From the last proposition, it is tempted
    for us to investigate the behavior of the fixed point
    $\bp$. Using the substitution
    $$\xb = (x_1,\dots,x_n) \
    \ \Longrightarrow  \ \xb = (x_1,\dots,1-(x_1+\dots+x_{n-1})).$$ We
    restrict ourselves to the consideration of the first $n-1$ coordinates of $V$.
    In this case, the found fixed point above is reduced to
    $\xb=(0,0,\dots,0)$. Moreover, using property (iv) in Theorem
    \ref{Prop_Pprties_V} and replacing $x_n = 1-(x_1+\dots+x_{n-1})$
    one can find the Jacobian at the fixed point $\bp$ namely,

    $$J\left[V(0,0,\dots,0)\right] =
    \begin{bmatrix}
    2P_{1n,1} & 0 & \dots & 0 \\
    2P_{1n,2} & 2P_{2n,2} & \dots & 0\\
    \vdots & \vdots & \ddots & \vdots \\
    2P_{1n,n-1}& 2P_{2n,n-1} & \dots & 2P_{n-1n,n-1}
    \end{bmatrix}
    .$$ Thus, the eigenvalues of $J(V(0,0,\dots,0))$ are
    $\{2P_{kn,k}\}_{k=1}^{n-1}$.

    \begin{thm}
        The fixed point $(0,0,\dots,0)$ is not repelling.
    \end{thm}

    From the discovered eigenvalues and Theorem \ref{thm_des_b-bstc_2D}, we conclude the next corollary.

    \begin{cor}\label{ccc1}
        If  $P_{lj,l} < \dfrac{1}{2}$ for all $l\in\{1,\dots,n-1\}$,
        $j\geq l+1,$ then the fixed point $\xb = (0,0,\dots,0)$ is
        attracting. \label{cor_fix_pt_att}
    \end{cor}

    The conditions of the last corollary does not imply the uniqueness of fixed point $ \bp $ as shown in the following example.
    \begin{exm}\label{exm_attrct_no_uni} Let us denote
        \begin{eqnarray}\label{notation_on_2D}
        P_{11,1} = A_1 \hspace{50pt} P_{13,1} = C_1 \hspace{50pt} P_{23,1} = E_1  \nonumber \\
        P_{11,2} = A_2 \hspace{50pt} P_{13,2} = C_2 \hspace{50pt}   P_{23,2} = E_2 \\
        P_{12,1} = B_1 \hspace{50pt} P_{22,1} = D_1 \hspace{50pt}P_{33,1} = F_1 \nonumber \\
        P_{12,2} = B_2  \hspace{50pt}P_{22,2} = D_2  \hspace{50pt} P_{33,2} = F_2 \nonumber
        \end{eqnarray}
        Take the heredity coefficients as follows: $ A_{1} = D_{2}=1, D_{1} = E_{1} = F_{1} =F_{2} =0  $ and $ B_{1}, C_{1}, E_{2} <\frac{1}{2} $
        such that $ B_{1} = C_{1} $ and $ C_{1} + C_{2} \leq \frac{1}{2}. $
        From \eqref{eqn_propty_HC}, one finds $ A_{2} = 0 $.
        Due to \cite[Theorem 5.3]{ME2015} a q.s.o generated by the given heredity coefficients is $b$- bistochastic.
        It is clear that, to obtain all fixed points, one must solve the following system
        \begin{eqnarray}\label{eqn_V_1=x_1_2D}
        \left\{
        \begin{array}{ll}
        {x_{{1}}}\left( \left( 1-2\,C_{1}\right) {x_{{1}}}+   2C_{1}-1 \right) = 0,
        \\[2mm]
        {x_{{1}}} \left( -2C_{2}  {x_{{1}}}+   \left( 2\,B_{2}-2\,C_{2}-2\,E_{2} \right) x_{{2}}+2\,C_{2} \right) +
        x_{{2}}\left(\left( 1-2\,E_{2} \right) {x_{{2}}}+  2\,E_{2}- 1 \right) = 0.
        \end{array}
        \right.  %\nonumber
        \end{eqnarray}

         From \eqref{eqn_V_1=x_1_2D}, to find solutions other than $ x_{1}=0 $, we have to solve
        $ \left( 1-2C_{1} \right) \left(  x_{1}-1\right) =0. $
        Thus, a possible solution is $ x_{1} =1 $. Hence, we have another fixed point $(1,0,0)$. Now assume $x_{1} =0 $, then
        from the second equation of \eqref{eqn_V_1=x_1_2D} one finds
        $ x_{{2}}\left(\left( 1-2\,E_{2} \right) {x_{{2}}}+  2\,E_{2}- 1 \right) =0
        $ which has the solutions $ x_{2} = 0$  or $ x_{2} =1$.
        Therefore, we have other fixed points $ (0,0,1) \textmd{ and } (0,1,0) $. Consequently, we infer that  $ (0,0,1) $ is not
        unique, while conditions of Corollary \ref{ccc1} are
        satisfied.
    \end{exm}

\section{Uniqueness of the Fixed Point }

 From the previous section, it is natural to ask: What are conditions for the uniqueness of the fixed point $ \bp
 $?. In this section, we are going to provide an answer to the
 raised question.

 In the sequel, we need an auxiliary fact.

    \begin{lem}
        The inequality
        \begin{eqnarray}
        A_1x_1+\dots+A_nx_n+C\leq (resp, <) 0
        \label{lem_inq_lin}
        \end{eqnarray}
        holds under conditions                                                                                                                                                                                                                                                                                                                                                                                                                                                                                                                                                                                                                             $0\leq x_1+\dots+x_n\leq1$, $x_k\geq 0$,
        $k\in\{1,\dots,n\}$ if and only if
        \begin{itemize}
            \item[(i)]$C\leq (resp, <)0 \ \ and$
            \item[(ii)]$A_k+C\leq (resp, <) 0, \ \ k=\overline{1,n} \ . $
        \end{itemize}
        \label{lem_lin_inql_leq0}
    \end{lem}

    The proof is obvious.

    Here is the main result of this section.

    \begin{thm}\label{thm_uni_fp}
        Let $ V:S^{n-1}\rightarrow S^{n-1} $ be a $ b- $bistochastic q.s.o. If
  \begin{equation}\label{UN11}
  P_{kk,k} < 1 \textmd{ and } P_{kj,k}<\dfrac{1}{2},
  \end{equation}
        for any $ k \in \{1,\dots, n-1\} $ and $ j=\{k+1,\dots,n\}. $ Then, the fixed point $ \bp $ is unique.
    \end{thm}

    \begin{proof}
        Let $\xb =(x_1,\dots,x_n) \in S^{2} $. Taking into account $ V(\xb)_{n} = 1- \left(V(\xb)_{1}+\dots+ V(\xb)_{n-1}\right) $, then all fixed points can be obtained by solving the system
        \[ \left\{
        \begin{array}{ll}
        V(\xb)_1 = x_1, \\
        \ \ \vdots\\
        V(\xb)_{n-1} = x_{n-1},
        \end{array} \right. \]
        which due to Theorem \ref{Prop_Pprties_V}, the last system is equivalent to
        \[ \left\{ \begin{array}{ll}
        0 &= \left( P_{11,1} - 2P_{1n,1} \right)x^2_1  + 2 \sum\limits_{j=2}^{n-1}\left(P_{1j,1} - P_{1n,1}\right)x_1x_j + 2 P_{1n,1}x_{1} - x_{1}, \\
        &\vdots\\
        0 &= \sum\limits_{l=1}^{k}\left( P_{ll,k} - 2P_{ln,k} \right)x_{l}^{2} + 2 \sum\limits_{l=1}^{k} \sum\limits_{j=l+1}^{n-1} \left( P_{lj,k} - P_{ln,k} \right)x_{j}x_{l} + 2 \sum\limits_{j=1}^{k}P_{jn,k} x_{j} - 2 \sum\limits_{l=1}^{k-1} \sum\limits_{j=l+1}^{k} P_{jn,k} x_{j}x_{l} - x_{k} \\
        &\vdots\\
        0 &= \sum\limits_{l=1}^{n-1}\left( P_{ll,n-1} - 2P_{ln,n-1} \right)x_{l}^{2} + 2 \sum\limits_{l=1}^{n-2} \sum\limits_{j=l+1}^{n-1} \left( P_{lj,n-1} - P_{ln,n-1} \right)x_{j}x_{l} + 2 \sum\limits_{j=1}^{n-1}P_{jn,n-1} x_{j}  \\
        & \ \ \ - 2 \sum\limits_{l=1}^{n-2} \sum\limits_{j=l+1}^{n-1} P_{jn,n-1} x_{j}x_{l} - x_{n-1}.
        \end{array} \right. \]

        Then for each $ k= 1,\dots, n-2$ the equation $ V(\xb)_{k} = x_{k} $ reduces to
        \begin{eqnarray}\label{eqn_fac_V(xb)_k=x_{k}}
        0 &=& x_{1}\left(
        \left( P_{11,k} - 2P_{1n,k} \right)x_{1} +
        2 \sum\limits_{j=2}^{n-1} \left( P_{1j,k} - P_{1n,k} \right)x_{j} +
        2P_{1n,k} - 2 \sum\limits_{j=2}^{k}P_{jn,k}x_{j} \nonumber
        \right)  \\
        &&  + x_{2}\left(
        \left( P_{22,k} - 2P_{2n,k} \right)x_{2} +
        2 \sum\limits_{j=3}^{n-1} \left( P_{2j,k} - P_{2n,k} \right)x_{j} +
        2P_{2n,k} - 2 \sum\limits_{j=3}^{k}P_{jn,k}x_{j}
        \right) + \dots \\ \nonumber
        && + x_{k}\left(
        \left( P_{kk,k} - 2P_{kn,k} \right)x_{k} +
        2 \sum\limits_{j=k+1}^{n-1} \left( P_{kj,k} - P_{kn,k} \right)x_{j} +
        2P_{kn,k} - 1
        \right). \nonumber
        \end{eqnarray}
      In the case of $ k=n-1 $, one has
      \begin{eqnarray}\label{eqn_fac_V(xb)_n-1=x_{n-1}}
      0 &=& x_{1}\left(
      \left( P_{11,k} - 2P_{1n,k} \right)x_{1} +
      2 \sum\limits_{j=2}^{n-1} \left( P_{1j,k} - P_{1n,k} \right)x_{j} +
      2P_{1n,k} - 2 \sum\limits_{j=2}^{k}P_{jn,k}x_{j} \nonumber
      \right)  \\
      &&  + x_{2}\left(
      \left( P_{22,k} - 2P_{2n,k} \right)x_{2} +
      2 \sum\limits_{j=3}^{n-1} \left( P_{2j,k} - P_{2n,k} \right)x_{j} +
      2P_{2n,k} - 2 \sum\limits_{j=3}^{k}P_{jn,k}x_{j}
      \right) + \dots \\ \nonumber
      && + x_{n-1}\left(
      \left( P_{n-1n-1,n-1} - 2P_{n-1n,n-1} \right)x_{n-1} +
       2P_{n-1n,n-1} - 1
      \right). \nonumber
      \end{eqnarray}
%      Due to the assumption, we find $ 2P_{kn,k} - 1 \neq 1 $ which means $ P_{kn,k} - 1 <\frac{1}{2} $ for any $ k \in \{1,\dots, n-1\}. $

%        Conversely,
        Let  $ P_{kk,k} < 1 \textmd{ and } P_{kj,k}<\frac{1}{2}$
        for any $ k \in \{1,\dots, n-1\} $, $ j=\{k+1,\dots,n\}. $
        Therefore, one finds that $ P_{11,1}-1<0 $ and $ 2P_{1j,1}-1<0 $, for any $ j\in\{2,\dots,n\} $. Hence, by means of Lemma \ref{lem_lin_inql_leq0} one gets
        \[
        \left( P_{11,1} - 2P_{1n,1} \right)x_1 + 2 \sum\limits_{j=2}^{n-1}\left(P_{1j,1} - P_{1n,1}\right)x_j + 2 P_{1n,1} - 1
        <0. \]
        This implies that, the solution $ x_{1} =0 $ is unique for $ V(\xb)_{1} = x_{1} $.
        %$ 2P_{1j,1} - 2P_{1n,1} + 2P_{1n,1 - 1} <0 $
        Using induction, let us assume $ x_{1} =0,  \dots, x_{k-1} =0 $ are unique solutions for $ V(\xb)_{1} = x_{1}, \dots, V(\xb)_{k-1} = x_{k-1} $,
        respectively. Now, we want to prove $ x_{k} =0 $ is also a unique solution for $ V(\xb)_{k} = x_{k} $.
        Using the assumption,
        %(i.e $ x_{1} =0,  \dots, x_{k-1} =0 $ are the unique solutions),
        it is clear that \eqref{eqn_fac_V(xb)_k=x_{k}} reduces to
        \[ x_{k}\left(
        \left( P_{kk,k} - 2P_{kn,k} \right)x_{k} +
        2 \sum\limits_{j=k+1}^{n-1} \left( P_{kj,k} - P_{kn,k} \right)x_{j} +
        2P_{kn,k} - 1
        \right) = 0 \]
        By the same argument as before, we know that
        \[
        \left( P_{kk,k} - 2P_{kn,k} \right)x_{k} +
        2 \sum\limits_{j=k+1}^{n-1} \left( P_{kj,k} - P_{kn,k} \right)x_{j} +
        2P_{kn,k} - 1 <0
        \]
        which gives the uniqueness of $ x_{k}=0 $.
        One may proceed the same procedure to show $ x_{n-1}=0 $ is unique.
        This completes the proof.
    \end{proof}
By $\mathcal{V}_u$ we denote the set of all $b$-bistochastic
q.s.o. whose heredity coefficients satisfy \eqref{UN11}. From
Theorem \ref{thm_uni_fp} we immediately find the following fact.

\begin{cor}
The set $\mathcal{V}_u$ is convex.
\end{cor}

This means that for any two operators $V_1,V_2$ taken from $
\mathcal{V}_u$, their convex combination $\l V_1+(1-\l)V_2$ also
has a unique fixed point $\mathbb{P}$, and moreover, all its
trajectory converges to $\mathbb{P}$. Note that, in general, this
fact is not true, for arbitrary q.s.o.

\begin{rem}
    The uniqueness conditions of fixed point $ \bp $ in Theorem \ref{thm_uni_fp} is only sufficient. Indeed,
    let us consider the following example in which we will keep the notations from Example \ref{exm_attrct_no_uni}.

    Let $ V $ be a q.s.o. generated by the following heredity coefficients
    $ A_{1},D_{2}<1, D_{1} = E_{1} = F_{1} =F_{2} =0  $, $ B_{1}, E_{2} \leq\frac{1}{2} $, $ C_{1}=\frac{1}{2} $ and $ C_{2} =0. $
%   From \eqref{eqn_propty_HC}, one finds $ A_{2} = 0 $.
    Due to \cite[Theorem 5.3]{ME2015}, $ V $ is $b$- bistochastic.
    It is clear that, to obtain all fixed points, we should  solve the following system
    \begin{eqnarray}\label{eqn_V_1=x_1_2D_not_uni}
    \left\{
    \begin{array}{ll}
    {x_{{1}}}\left( \left( A_{1}-1 \right)x_{1} + \left( 2B_{1}-1\right) {x_{{2}}} \right) = 0, \nonumber
    \\[2mm]
    {x_{{1}}} \left(  A_{{2}} {x_{{1}}}+   \left( 2\,B_{{2}}-2\,E_{{2}} \right) x_{{2}} \right) +
    x_{{2}}\left(\left( D_{{2}}-2\,E_{{2}} \right) {x_{{2}}}+  2\,E_{{2}}- 1 \right) = 0. \label{eqn_V_2=x_2_2D}
    \end{array}
    \right.  %\nonumber
    \end{eqnarray}
    From the last system, it is evident that we need to solve $ \left( A_{1}-1 \right)x_{1} + \left( 2B_{1}-1\right) {x_{{2}}} =0 $ in order to find
    solutions other than $ x_{1}=0 $. So, assuming $ x_{1}>0 $, we get
    \[ x_{1} = \dfrac{x_2(1-2B_{1})}{A_{1}-1} \]
    which it turns out that $ x_{1} $ does not belong to the simplex $ S^{2} $ (i.e $ x_{1}<0 $ since $ A_{1}<1 $).
    Whence, the only possible solution is $ x_{1}=0 $.
    Using this fact, one gets $ x_{{2}}\left(\left( D_{{2}}-2\,E_{{2}} \right) {x_{{2}}}+  2\,E_{{2}}- 1 \right) = 0 $, which yields
    \[ x_{2} = \dfrac{1-2E_{2}}{D_{2} -2 E_{2}} \]
    Due to $ D_{2}<1 $, we infer that $ x_{2}\leq 0 $, this means that $ x_{2}=0 $ is only solution.
    Hence, the fixed point $ (0,0,1) $ is unique, while conditions \eqref{UN11} in Theorem \ref{thm_uni_fp} are not satisfied .

\end{rem}

    \section{Uniqueness vs Contraction}
    It is a well-known that a strict contractive operator $V$ has a unique fixed point, and all trajectories of $V$ converge to that point with exponential rate.
    Hence, it is interesting to know whether $ b- $bistochastic q.s.o. with unique fixed point is a strict contraction.
    In this section, we show it is not so.
    For the sake of simplicity, in this section, we restrict ourselves to one and two dimensional cases, respectively.

    Let us recall that
    a stochastic operator $V$ is said to be \textit{strict contraction} if there exist $\a \in [0,1)$ such that
    \[ \|V(\xb) - V(\yb)\| \leq \a \| \xb - \yb \|, \ \ \xb,\yb \in S^{n-1}. \]
    Here, $ \| \xb\| = \sum_{i=1}^{ n} |x_i| $.

It is known \cite{K} the following fact.

\begin{thm}\label{1SC} A q.s.o.
$V:S^{n-1} \rightarrow S^{n-1}$ is a strict contraction if and
only if
        $$\max_{i_{1}, i_{2}, k}\sum_{j=1}^{n}|p_{i_{1}k,j} -
        p_{i_{2}k,j}|<1.
        $$
    \end{thm}

    From this theorem we immediately find the following results.

    \begin{thm}\label{thm_strt_con_1D}
        Let $V:S^{1}\rightarrow S^{1}$ be a $b-$bistochastic q.s.o., then $V$ is a strict contraction if and only if
        $$ \max\{P_{12,1}; |P_{11,1}-P_{12,1}| \} < \dfrac{1}{2}. $$
    \end{thm}

    \begin{thm} \label{thm_strt_con_2D}
        Let $V: S^{2} \rightarrow S^{2}$ be a $b-$bistochastic q.s.o. Then $V$ is a strict contraction
        if and only if
        the maximum value of the following quantities is strictly less than 1.
        \begin{itemize}
            \item[(a)] $ \left| A_{{1}}-B_{{1}} \right| + \left| A_{{2}}-B_{{2}} \right| +
            \left| A_{{1}}+A_{{2}}-B_{{1}}-B_{{2}} \right|
            $
            \item[(b)]$ B_{{1}}+ \left| B_{{2}}-D_{{2}} \right| + \left| B_{{1}}+B_{{2}}-D_{{2
                }} \right|
                $
                \item[(c)]$C_{{1}}+ \left| C_{{2}}-E_{{2}} \right| + \left| C_{{1}}+C_{{2}}-E_{{2
                    }} \right|
                    $
                    \item[(d)]$\left| A_{{1}}-C_{{1}} \right| + \left| A_{{2}}-C_{{2}} \right| +
                    \left| A_{{1}}+A_{{2}}-C_{{1}}-C_{{2}} \right|
                    $
                    \item[(e)]$B_{{1}}+ \left| B_{{2}}-E_{{2}} \right| + \left| B_{{1}}+B_{{2}}-E_{{2
                        }} \right|
                        $
                        \item[(f)]$2\,C_{{1}}+2\,C_{{2}} $

                        \item[(g)]$ \left| B_{{1}}-C_{{1}} \right| + \left| B_{{2}}-C_{{2}} \right| +
                        \left| B_{{1}}+B_{{2}}-C_{{1}}-C_{{2}} \right|
                        $
                        \item[(h)]$2\, \left| D_{{2}}-E_{{2}} \right|  $
                        \item[(i)]$2\,E_{{2}} $
                    \end{itemize}
                    where the notations of heredity coefficients used here are the same as given by \eqref{notation_on_2D}.
                \end{thm}

                In both cases (i.e one and two dimensional simplices), the uniqueness of the fixed point of $ b- $bistochastic q.s.o.
                does not imply their strict contractivity. Indeed,
                we consider the following examples.

                \begin{exm}
                    Let $ V $ be given by the coefficients $ P_{11,1}=\frac{2}{3} $ and $ P_{12,1}=0 $.
                    Due to  $2(P_{11,1}-P_{12,1}) = 2(\frac{2}{3}) > 1$ one concludes that $ V $ is not a strict contraction.
                    Yet, it is still a $b-$bistochastic q.s.o. and it converges to a unique fixed point $ (0,1) $
                    (see Theorem \ref{thm_uni_fp}).\\ %ne has
                    %$$\lim\limits_{m \rightarrow \infty } V^{(m)}(\xb) = (0,1).$$
                \end{exm}

                \begin{exm}
                   Let us consider a q.s.o. $ V:S^{2}\rightarrow S^{2} $ given
                   by the heredity coefficients
                    \[ A_{1} = 0 \ \ B_{1} = 0 \ \ C_{1} = 0  \ \ D_{1} = 0 \ \ E_{1} = 0 \ \ F_{1} = 0, \]
                    \[ A_{2} = 1 \ \ B_{2} = 0  \ \ C_{2} = 0 \ \ D_{2} = 0 \ \ E_{2} = 0 \ \ F_{2} = 0. \]

                   According to \cite[Theorem 5.3]{ME2015} the defined operator is $b-$bistochastic q.s.o.  Moreover, one finds that
                    \[
                    \left| A_{{1}}-B_{{1}} \right| + \left| A_{{2}}-B_{{2}} \right| +
                    \left| A_{{1}}+A_{{2}}-B_{{1}}-B_{{2}} \right| = 2
                    \]
                    which implies $ V $ is not a strict contraction. On the other hand, from Theorem \ref{thm_uni_fp} one concludes that
                    V has a unique fixed point $ (0,0,1) $.
            \end{exm}

                \section{Mixing property of nonhomogenous Markov chains associated with $ b- $Bistochastic q.s.o}

                In this section, we first construct non-homogeneous Markov measures associated with
                $b$-bistochastic q.s.o. Then we study its mixing property.

                Now, let us recall some  terminologies.
                Denote
                $ \O = \prod\limits_{i=0}^{\infty} E_i\ , \ E_i= \{1,\dots, n\} $. %= \{\w= (\w_i)| \w_i\in E\}.$$
                A subset of $ \O $ given by
                %$$A(i_1,i_2,\dots, i_k) = \left\{ \w \in \S | \w_{i_1} = i_1, \dots, w_{i_k} = i_k \right\}$$
                \[  A^{[l,m]}(i_{l},\dots,i_{m}) = \{ \w \in \O \ | \ \w_l = i_l, \dots, \w_m = i_m\} , \]
                %\[ A_k^{(m)}(i_k,i_{k+1},\dots, i_m) = \{ \w \in \O \ \vert\  \w_k = i_k,\w_{k+1} = i_{k+1} \dots \w_{m} = i_m \}. \]
                is called \textit{thin cylindrical} set, where $ i_{k} \in \{ 1,\dots,n \} $. By $ \gf $ we denote the $\s-$algebra generated by thin cylindrical sets.
                %The $ \s- $ algebra $ \gf $ is also generated by what shall be called \textit{thin} cylinders, that is sets of the form $ \{ \w:w_k = i_k \ \vert \ l \leq k \leq m \} $, where the $ i_k \in \{1,\dots, n\} $. Indeed, each cylinder is a finite, disjoint union of thin cylinders.
                Since the finite disjoint unions of thin cylinders form an algebra which generates $ \gf $, therefore a measure $ \m $ on $ \gf $ is uniquely deteremined by the values
                \begin{eqnarray}\label{eqn_mea_gen}
                \m \left( A^{[l,m]}(i_{l},\dots,i_{m}) \right) =: \m^{[l,m]} \left( A^{[l,m]}(i_{l},\dots,i_{m}) \right).
                \end{eqnarray}
                              Therefore, to define a measure on measurable space $ (\O,\gf) $, it is enough to define on the thin cylindrical sets.
                %Let $ \bp^{[m,m+1]} = \left( p_{i,j}^{[m,m+1]} \right)_{i,j \geq 0}$ be a stochastic matrix. Now we want to consider a measure $ \m $ on $ (\O, \gf) $.
                Recall that, a measure $ \m $ on $ (\O,\gf) $ is called \textit{non-homogeneous Markov measure}
                if for any sequence of states $ i,j,k_{m-1}, k_{m-2}, \dots, k_{0} $ and for any $ m \geq 0, \ m\in\bn $ one has
                \[ \m\left( \w_{m+1} =j \left| \w_{m} = i, \w_{m-1} = k_{m-1}, \dots, \w_{0} = k_{0} \right.  \right) = \m\left( \w_{m+1} =i \left| \w_{m} = j \right. \right) =: p_{i,j}^{[m,m+1]}\]
                 In a convenient way, the collection of transition probabilities at instance $ m $ can be represented in a square stochastic matrix known as \textit{transition probability matrix} $ \bp^{[m,m+1]} = \left( p_{i,j}^{[m,m+1]} \right)_{i,j \geq 0}$.
                The sequence of such matrices  $ \textbf{P} := \{\bp^{[m,m+1]}\}_{m\geq 0}$ is called a
                \textit{ non-homogeneous Markov chain}. Note that every matrix $ \bp^{[m,m+1]} $ acts on $ \xb = (x_1,x_2,\dots, x_n) \in S^{n-1} $ by
                \[ \left( \bp^{[m,m+1]} (\xb) \right)_k = \sum\limits_{i=1}^{d} x_ip_{i,k}^{[m,m+1]}, \ k=\overline{1,d} \]
                Moreover, for any natural number $k,m \  (k > m \geq 0) $, the transition  probabilities  in $ k-m $ steps $ \bp^{[m,k]} $  can be
                calculated as a composition of linear operators,
                i.e.
                \[ \bp^{[m,n]} = \bp^{[m,m+1]} \circ \bp^{[m+1,m+2]} \circ \dots \circ \bp^{[k-1,k]}. \]

Now, we are going to give a construction of a Markov measure
associated with q.s.o. Note that, this kind of construction was
first considered in
\cite{GanistogenQO,ganisarymsakovcenlimthmqc}.
                Certain properties of the associated Markov chains have been investigated in several paper
                such as \cite{pulkaMixnErPrNonMC,farL1ErNonDiscMP,farSupAkmaMarPQSP}

                Let $V:S^{n-1}\rightarrow S^{n-1}$ be a q.s.o. defined by heredity coefficients $\{P_{ij,k}\}_{i,j,k=1}^{n} $ and we denote
                $\xb_j^{(m)} = (V^{(m)}(\xb))_j, \  \xb \in S^{n-1}. $
                Denote $ E = \{1,2,\dots, n\} $ and consider $ (\O,\gf)
                $. Define a matrix $ \bh_{V,\xb}^{[k,k+1]} = \left( H_{ij,\xb}^{[k,k+1]} \right)_{i,j = 1}^{n} $ by
                \begin{eqnarray}\label{eqn_H}
                H_{ij,\xb}^{[k,k+1]} = \sum\limits_{l=1}^{n}P_{il,j}\xb_l^{(k)}, \ \xb=(x_1,\dots,x_n)\in S^{n-1}.
                \end{eqnarray}
                One easily can prove that  $ \bh_{V,\xb}^{[k,k+1]} $ is a stochastic matrix.\\
                Now, we define a measure associated with a q.s.o. $ V $ by
                \begin{eqnarray}\label{eqn_def_mea_on_cylnder_set}
                \m_{V,\xb}^{[k,m]}(A^{[k,m]}(i_k,i_{k+1},\dots, i_m)) &=& x_{i_k}^{(k)} H_{i_k,i_{k+1} ,\xb}^{[k,k+1]}H_{i_{k+1},i_{k+2},\xb}^{[k+1,k+2]}\dots H_{i_{m-1},i_m,\xb}^{[m-1, m]}.
                %& = & x_{i_n}^{(m)} H_.
                \end{eqnarray}
                The defined measures $ \m_{V,\xb}^{[k,m]} $ are non-homogeneous Markov measures (see \cite{GanistogenQO,MG2015}).
                \begin{prop}
                    The measure $ \m_{V,\xb}^{[k,m]} $ given by \eqref{eqn_def_mea_on_cylnder_set} defines a non-homogeneous Markov measure $ \m_{V,\xb} $.
                \end{prop}

%\begin{defn}\label{defn_mixing}
 %                   A measure $ \m $ defined on $\Omega$ is said to satisfy \textit{mixing property} if for any $ A,B \in \gf $ one has
  %                  \[ \lim\limits_{m \rightarrow \infty} \left\vert \m  \{ \w_k =i , \w_{m} = j \}  - \m \{ \w_{k}=i\} \m \{\w_{m}= j\} \right\vert  = 0, \]
   %                 for any $ k \in \bn, m > k $.
    %            \end{defn}

We denote $ \s $ as the shift transformation of $\Omega$, i.e. $
\s(\w)_{n} = \w_{n+1}$, $\w\in\Omega$.

\begin{defn}\label{defn_mixing}
    A measure $ \m $ defined on $\Omega$ is said to satisfy \textit{mixing property} if for any $ A,B \in \gf $ one has
    \[ \lim\limits_{m \rightarrow \infty} \left\vert \m  \left( A \cap \s^{m}(B) \right)   - \m \left( A \right) \m \left( \s^{m}(B) \right)  \right\vert  = 0, \]
    for any $ k \in \bn, m > k $.
\end{defn}

                \begin{rem}
                    If the measure $ \m $ is shift invariant, ($  \m\left( \s^{-1}(A) \right) = \m(A)  $ for any $ A \in \gf $,
                    then the mixing reduces to the well-known notion of mixing \cite{SKF}.
                \end{rem}

                Denote a set $ A_k^{[m]}(i,j) = \{\w \in \O | \w_{k} = i, \w_{m} =j\}\nonumber $.
                Then, for the measure $ \m_{V,\xb} $ (see \ref{eqn_def_mea_on_cylnder_set}) we have
                \begin{eqnarray}
                \m_{V,\xb} (A_k^{[m]}(i,j)) = x_i^{(k)}H_{ij,\xb}^{[k,m]}, \ \ \xb \in (x_1,x_2,\dots, x_n)
                \end{eqnarray}
                and the rule of $H_{ij,\xb}^{[k,m]}$ can be calculated as follows
                \begin{eqnarray}\label{eqn_P_iterate}
                H_{ij,\xb}^{[k,m]} = \sum\limits_{l_{k+1},\dots,l_{m-1} =1}^{n}H_{il_{k+1},\xb}^{[k,k+1]}H_{l_{k+1}l_{k+2},\xb}^{[k+1,k+2]}\dots H_{l_{m-1}j,\xb}^{[m-1,m]}\nonumber
                \end{eqnarray}
                or it can be computed as a usual matrix multiplication (let $\bh_{V,\xb}^{[k,m]} = \left(H_{ij,\xb}^{[k,m]}\right)_{i,j = 1}^{n}$)
                \begin{eqnarray}\label{eqn_mul_H}
                \bh_{V,\xb}^{[k,m]} = \bh_{V,\xb}^{[k,k+1]} \circ \bh_{V,\xb}^{[k+1,k+2]}\circ \dots \circ \bh_{V,\xb}^{[m-1,m]}.
                \end{eqnarray}

By support of $ \xb $ we mean a set $ Supp(\xb) = \left\{ i
\in E \vert x_{i} \neq 0 \right\}. $

                \begin{defn} A q.s.o.  $V$ is called
                    {\it asymptotically stable} \index{asymptotically stable} (or \textit{regular}) \index{regular} if there exists a
                    vector $\pb^*\in~S^{n-1}$ such that for all $\xb \in S^{n-1} $ one
                    has
                    $$\lim_{m\to\infty} \left\| V^{(m)}(\xb) -
                    \pb^* \right\| = 0,$$
                    where the norm $ \norm{\xb} = \sum\limits_{i =0}^{n} \abs{x_{i}} $.
                \end{defn}

                In what follow we need an auxiliary result from \cite{bolka_pulka_mix_cls_qso}.

                \begin{thm}\label{thm_asym_stb_qso}\cite{bolka_pulka_mix_cls_qso}
                    Let  $V$ be a q.s.o. The following statements are equivalent:
                    \begin{enumerate}
                        \item[(i)]  $V$ is asymptotically stable;
                        \item[(ii)]
                        there exists
                        $\pb^* \in S^{n-1} $ such that for all $\xb\in S^{n-1}$
                        and $\zb
                        \in S^{n-1} $ with $Supp(\zb)\subseteq Supp(\xb)$  we have
                        $$\lim\limits_{m\to\infty} \left\| H^{[0,m]}_{V,\xb} (\zb) - \pb^* \right\| = 0;$$
                        \item[(iii)] there exists
                        $\pb^*\in S^{n-1} $ such that for all $k\geq 0,$ and all $\xb\in
                        S^{n-1},$ $\zb \in S^{n-1} $ with $Supp(\zb)\subseteq Supp(\xb)$
                        we have
                        $$\lim\limits_{m\to\infty} \left\| H^{[k,m]}_{V,\xb} (\zb)
                        - \pb^* \right\| = 0.$$
                    \end{enumerate}\label{thm_strong_as}
                \end{thm}

                Now, let us consider any $ b- $bistochastic q.s.o. $ V $ with a unique fixed point $ \bp $ .
                Due to Theorem \ref{thm_lim_pt_e} we conclude that for any $ \xb \in S^{n-1} $
                \[ \lim\limits_{m \rightarrow \infty} V^{(m)}(\xb) = \bp, \]
                (i.e $ V $ is regular). Furthermore, we consider the corresponding Markov measure $ \m_{V,\xb} $.

A main result of this section is the following theorem.

                \begin{thm}
                    Let $ V $ be a $ b- $bistochastic q.s.o. with a unique fixed point $ \bp $. Then, for every $ \xb \in ri \ S^{n-1} $,
                    the measure $ \m_{V,\xb} $ satisfies the mixing property.
                \end{thm}

                \begin{proof}
                    Since the algebra $ \gf  $ is generated by thin cylindrical set, due to density argument, it is enough to proof the theorem for the thin cylindrical set.
                    Let $ A = A^{[k,l]}(i_{k}, \dots, i_{l}) $ and $ B=B^{[s,t]}(j_{s},\dots, j_{t}) $
                    and
                    \begin{eqnarray}\label{eqn_mea}
                    \t_{m} := \left\vert \m  \left( A \cap \s^{m}(B) \right)   - \m \left( A \right) \m \left( \s^{m}(B) \right)  \right\vert.
                    \end{eqnarray}
                    It is clear that
                    \begin{eqnarray}
                    \m\left( A \cap \s^{m}(B) \right) & = & \xb_{i_{k}}^{(k)}H_{i_{k}i_{k+1},\xb}^{[k,k+1]}\cdots H_{i_{l-1}i_{l},\xb}^{[l-1,l]}
                    %\cdot
                    H_{i_{l}j_{s},\xb}^{[l,s+m]}
                    H_{j_{s}j_{s+1},\xb}^{[s+m,s+m+1]} \cdots
                    H_{j_{t-1}j_{t},\xb}^{[t-1+m,t+m]} \label{eqn_m_inter} \\
                    \m \left( A \right) \m \left( \s^{m}(B) \right) & = & \xb_{i_{k}}^{(k)}H_{i_{k}i_{k+1},\xb}^{[k,k+1]}\cdots H_{i_{l-1}i_{l},\xb}^{[l-1,l]}
                    %\cdot
                    \xb_{j_{s}}^{(s+m)}
                    H_{j_{s}j_{s+1},\xb}^{[s+m,s+m+1]} \cdots
                    H_{j_{t-1}j_{t},\xb}^{[t-1+m,t+m]} \label{eqn_m_multip}
                    \end{eqnarray}
               From \eqref{eqn_m_inter} and \eqref{eqn_m_multip} we find
               \[ \t_{m} = \xb_{i_{k}}^{(k)}H_{i_{k}i_{k+1},\xb}^{[k,k+1]}\cdots H_{i_{l-1}i_{l},\xb}^{[l-1,l]}
               H_{j_{s}j_{s+1},\xb}^{[s+m,s+m+1]} \cdots
               H_{j_{t-1}j_{t},\xb}^{[t-1+m,t+m]}
               \left\vert
               H_{i_{l}j_{s},\xb}^{[l,s+m]} - \xb_{j_{s}}^{(s+m)}
               \right\vert
                \]
               Due to stochasticity, one infers that
               \[ \t_{m} \leq \left\vert
               H_{i_{l}j_{s},\xb}^{[l,s+m]} - \xb_{j_{s}}^{(s+m)}
               \right\vert \]
                    Asymptotic stability of $ V $ and Theorem \ref{thm_asym_stb_qso} implies
                    \[ \lim\limits_{m \rightarrow \infty }H_{V,\xb}^{[k,m]}(\eb_{i}) = \bp \textmd{ and } \lim\limits_{m \rightarrow \infty }\xb^{(m)} = \bp \]
                    Therefore, $ \lim\limits_{m \rightarrow \infty }\left\vert
                    H_{i_{l}j_{s},\xb}^{[l,s+m]} - \xb_{j_{s}}^{(s+m)}
                    \right\vert =0  $, which gives $ \lim\limits_{m \rightarrow \infty } \t_{m} =0 $. This completes the proof.
                \end{proof}

                \section{Absolute Continuity of Corresponding Markov Measures}

In the previous section, we have defined non-homogeneous Markov
measures associated with q.s.o. It is clear that that these
measures depend on the initial state of $\xb$. Given q.s.o. it is
interesting to know how the measures $\m_{V,\xb}$ relate to each
other for different initial states. In this section, we are going
to examine the absolute continuity of these measures for
$b$-bistochastic q.s.o.

Let us recall some necessary notions and notations. Let $(\O,
\mathfrak{F})$  be a measurable space as introduced before. Assume
that the space is equipped with  two different probability
measures $\m$ and $\bar{\m}$. In addition, suppose on the space
given there is defined a family $(\mathfrak{F}_m)_{m-1}$ of
$\s-$algebras such that $\gf_1 \subseteq \gf_2 \subseteq \dots
\subseteq \gf$ and
                \begin{eqnarray}
                \gf = \s\left( \bigcup\limits_{i=1}^{\infty} \gf_i \right)
                \end{eqnarray}
Moreover, for each $\m_m$ and $\bar{\m}_m$, we restrict the
measure on the defined space $ (\O, \gf_m)$. Let $B \in \gf_m$ and
$  \m_m(B) = \m(B), \ \  \bar{\m}_m(B) = \bar{\m}(B) $.
                %\[ \m_m(B) = \m(B) \ \ \ \ \bar{\m}_m(B) = \bar{\m}(B)\]

\begin{defn}
Let $\m$ and $\bar{\m}$ be two probabilistic measures. It is
called that $\bar{\m}$ is \textit{absolutely continuous} with
respect to $\m$ (write it as $\bar{\m} \ll \m$) if $\bar{\m}(A) =
0$ whenever $\m(A) = 0, A\in \gf$. In addition, if $\m \ll
\bar{\m}$ and $\bar{\m} \ll \m$, then $\m$ and $\bar{\m}$ are
equivalent (in short $\bar{\m} \ \sim \ \m$)

Besides, $\m$ and $\bar{\m}$ are \textit{singular} (or
\textit{orthogonal} $ (\bar{\m} \perp \m) $) if there is a set $A
\in \gf $ such that $\bar{\m} (A)=1$ and $\m(A^c) = 1 \ $
                \end{defn}

                %Definition above gives general definition on the space $(\O, \gf)$, thus it is natural to provide componentwise definition on $(\O, \gf_m)$.

                \begin{defn} A measure
                    $\bar{\m}$ is \textit{locally absolutely continuous} with respect to measure $\m$ (notation, $\bar{\m}\ll^{loc}\m$) if $ \bar{\m}_m \ll \m_m \ \ \textmd{for every} \ \ m\geq1 $
                    \[ \bar{\m}_m \ll \m_m \ \ \textmd{for every} \ \ m\geq1 \]
                \end{defn}

                In the sequel, we will use the following result.
                \begin{thm} \cite{shir}
                    )
                    \label{thm_test_re_m}
                    Let there exist constants $A$ and $B$ such that $0 < A < 1, B > 0$ and
                    \begin{eqnarray} \label{thm_eqn_con_test}
                    \m\{ 1-A \leq \a_m \leq 1+B \} =1, \ \ m\geq 1
                    \end{eqnarray}
                    Then if $\bar{\m} \ll^{loc} \m$ we have
                    \[ \bar{\m} \ll \m \Leftrightarrow \bar{\m}\left\{  \sum\limits_{m=1}^{\infty} E \left( \left. (1- \a_m)^2 \right\vert \mathfrak{F}_{m-1} \right) < \infty \right\}  =1 \]
                    \[ \bar{\m} \perp \m \Leftrightarrow \bar{\m}\left\{  \sum\limits_{m=1}^{\infty} E \left( \left.  (1- \a_m)^2 \right\vert \mathfrak{F}_{m-1} \right) = \infty \right\}  =1 \]
                    where $\a_m = \frac{z_m}{z_{m-1}}$, and  $z_m$ is given by $ z_m(i_0,i_1,\dots, i_m) = \dfrac{\bar{\m}(i_0,i_1,\dots, i_m)}{\m(i_0,i_1,\dots, i_m)}.$
                \end{thm}

                It worth to note the following remark.

                \begin{rem}\label{rem_exptn}
                    Let us assume $ \left( 1-\a_m \right)^{2}$ can be decomposed in the following form:
                    \begin{eqnarray}
                    \left( 1-\a_m \right) &=& r_1\Bc_{\{ A_1 \}} + r_2\Bc_{\{ A_2 \}} + \dots + r_k\Bc_{\{ A_k \}}, \end{eqnarray}
                    %where
                    %\[ \gf_{m} = A_{1} \cup A_{2} \cup \dots \cup A_{n}. \]
                    %and
                    %\begin{eqnarray}
                    %\gf_{m-1} &=& \hat{A}_1^{(m-1)} \cup \hat{A}_2^{(m-1)} \cup \dots \cup \hat{A}_k^{(m-1)} \nonumber
                    %\end{eqnarray}
                    Consequently, the expectation can be written explicitly as follows
                    \begin{eqnarray}
                    E(\left(1-\a_{m}\right)^{2}|\gf_{m-1}) &=& E\left(  \left. \sum\limits_{i=1}^{k} r_i \Bc_{\{A_i \}} \right\vert \gf_{m-1} \right) \nonumber \\
                    & = & \sum\limits_{i=1}^{k}r_i E\left( \left. \Bc_{\{A_i\}} \right\vert \gf_{m-1} \right) \nonumber
                    =  \sum\limits_{i=1}^{k} r_i \bar{\m} \left(\left. A_i \right\vert \gf_{m-1}\right). \nonumber
                    \end{eqnarray}
                    In what follows, it is assumed that every $\s$-algebra $ \gf_{m-1} $ is finite, i.e. generated by finitely many random variables
                    $ \hat{A}_1, \hat{A}_2, \dots,
                    \hat{A}_l$. Then  the expectation with respect to $ \s- $algebra $ \gf_{m-1} $ is calculated by
                    \begin{eqnarray}
                    \bar{\m}(A_i|\gf_{m-1}) = \sum\limits_{j=1}^{l} \bar{\m}(A_i|\hat{A}_j) \ ; \ \ \  \bar{\m}(A_i|\hat{A}_j) = \dfrac{\bar{\m}(A_i \cap \hat{A}_j)}{\bar{\m}(\hat{A}_j)}. \nonumber
                    \end{eqnarray}
                    Therefore, one needs to compute the conditional measures with respect to all components $A_i$ and $\hat{A}_j$  in order to find the expectation.
                \end{rem}

          In this section, for the sake of simplicity, we consider a class q.s.o. $V_a:S^{1}\rightarrow S^{1}$ defined by
                \begin{eqnarray}\label{eqn_chsn_qso}
                P_{11,1}=a, \ P_{21,1}= P_{12,1}=0, \ P_{22,1}=0.
                \end{eqnarray}
                 Thus, $ P_{11,2}=1-a, P_{21,2} =P_{12,2}=1 $ and $ P_{22,2} =1 $. One easily can check $ V_a $ is indeed a class of
                 $b-$bistochastic q.s.o. and if $ 0\leq a <1 $ then $ V_{a} $  has unique fixed point $(0,1)$.
                 %which has unique fixed point $(0,1)$.
For this q.s.o. we have $E = \{1,2\}$, $ \xb = (x_1,x_2) \in S^{1}
$. One can calculate that
                \begin{eqnarray}\label{eqn_val_H}
                H_{11,\xb}^{[k,k+1]} = (ax_1)^{2^{k}}, \ \ H_{12,\xb}^{[k,k+1]} = 1- (ax_1)^{2^{k}}, \ \
                H_{21,\xb}^{[k,k+1]} = 0, \ \  H_{22,\xb}^{[k,k+1]} = 1.
                %\nonumber
                \end{eqnarray}

                \begin{prop}\label{prop_comp_mue}
                    Let $V_a$ be a $b-$bistochastic q.s.o. given by \eqref{eqn_chsn_qso}. Assume that $ \xb = (x_1,x_2) \in S^{1} $
                    and let $\m_{\xb,V_{a}}$ be a measure defined by \eqref{eqn_def_mea_on_cylnder_set}, then the following statements hold:
                    \begin{itemize}
                        \item[(i)] $ \m_{\xb,V_a}\left( A^{[l,m]}(1,\dots,1) \right) =a^{2^{m}-2^{l-1}}x_{1}^{2^{m}} $.
                        \item[(ii)] $\m_{\xb,V_a}\left( A^{[l,m]}(2,\dots,2) \right) =1-a^{2^{l-1}}x_{1}^{2^{l}} $.
                        \item[(iii)] $ \m_{\xb,V_a}\left( A^{[k,k+1]}(2,1) \right) =0 $ \textmd{for every} $ k\in\bn $.
                        \item[(iv)] $ \m_{\xb,V_a}\left( A^{[l,m]}  ( \underbrace{1,\dots,1}_{k+1-l}, 2,\dots,2 )\right)
                        = a^{2^{k}-2^{l-1}}x_{1}^{2^{k}}(1-(ax_1)^{2^{k}})$ for every $ k \in \{ l,l+1,\dots,m-1 \} $
                    \end{itemize}
                \end{prop}
                The proof immediately follows from \eqref{eqn_val_H}.

                The following theorem is a main result of this section.

                \begin{thm}
                    Let $V_a$ be a $b-$bistochastic q.s.o. given by \eqref{eqn_chsn_qso} and
                    a measure $ \m_{\xb,V_{a}} $ be defined by \eqref{eqn_def_mea_on_cylnder_set}, then the following statements hold:
                    \begin{itemize}
                        \item[(i)] $ \m_{\xb,V_a} \sim \m_{\yb,V_a} $
                        \item[(ii)] $ \m_{\xb,V_{a_1}} \sim \m_{\xb,V_{a_2}}$
                        \item[(iii)] $ \m_{\xb,V_{a_1}} \sim \m_{\yb,V_{a_2}}$
                    \end{itemize}
                    where $ \xb = (x_1,x_2), \yb =(y_1,y_2) \in S^{1} $ and $ 0 \leq a,a_1,a_2\leq 1 $.%, \ a_1 \neq a_2 $.
                     \end{thm}

                \begin{proof}  It is enough to prove (i), since the other cases can be proceeded by the same argument.
                    First, let us denote $ z_m = \frac{\m_{\xb,V_{a}}}{\m_{\yb,V_{a}}}$  and $\a_{m}=\frac{z_{m}}{z_{m-1}} $.
                    %\[ z_m = \dfrac{\m_{\xb,V_{a}}}{\m_{\yb,V_{a}}}, \ \ \ \a_{m}=\dfrac{z_{m}}{z_{m-1}}. \]
                    By definition, we let $\frac{0}{0}:=1$. By virtue of Proposition \ref{prop_comp_mue}, $ z_{m}(\w) $ can be written
                    as follows
                    \begin{eqnarray}
                    z_m(\w) &=&  \left(\dfrac{x_1}{y_1}\right)^{2^{m}} \Bc_{\{A^{[l,m]}(1,\dots,1)\}} +
                    \dfrac{\left(1-a^{2^{l-1}}x_1^{2^{l}}\right)}{\left(1-a^{2^{l-1}}y_1^{2^{l}}\right)} \Bc_{\{A^{[l,m]}(2,\dots,2)\}} \nonumber \\
                    &&+\sum\limits_{k=1}^{m-1} \Bc_{\{A^{[k,k+1]}(2,1)\}} +\sum\limits_{k=1}^{m-1} \dfrac{(x_1)^{2^{k}} (1-(ax_1)^{2^{k}})}{(y_1)^{2^{k}} (1-(ay_1)^{2^{k}})}\Bc_{\{A^{[l,m]}_{k}  (
                        1,\dots, \underbrace{1}_{k^{th}}, 2,\dots,2)\}}. \nonumber
                    \end{eqnarray}

                    Accordingly, $z_{m-1}$ is
                    \begin{eqnarray}
                    z_{m-1}(\w) &=&  \left(\dfrac{x_1}{y_1}\right)^{2^{m-1}} \Bc_{\{ A^{[l,m-1]}(1,\dots,1) \}} +
                    \dfrac{\left(1-a^{2^{l-1}}x_1^{2^{l}}\right)}{\left(1-a^{2^{l-1}}y_1^{2^{l}}\right)} \Bc_{\{\{A^{[l,m-1]}(2,\dots,2) \}} \nonumber \\
                    &&+\sum\limits_{k=1}^{m-2} \Bc_{\{A^{[k,k+1]}(2,1)\}} +\sum\limits_{k=1}^{m-2} \dfrac{(x_1)^{2^{k}} (1-(ax_1)^{2^{k}})}{(y_1)^{2^{k}} (1-(ay_1)^{2^{k}})}\Bc_{\{ A^{[l,m-1]}_{k}  (
                        1,\dots, \underbrace{1}_{k^{th}}, 2,\dots,2)\}}. \nonumber
                    \end{eqnarray}

                    Let $ A^{[k,k+1]}(i_{k},i_{k+1}) \cap A^{[m]}(i_{m}) = A^{[k,k+1,m]}(i_{k}, i_{k+1}, i_{m}) $.
                    One can see that $z_m$ and $z_{m-1}$, respectively, can be rewritten
                    as follows
                    \begin{eqnarray}
                    z_m(\w) &=&  \left(\dfrac{x_1}{y_1}\right)^{2^{m}} \Bc_{\{A^{[l,m]}(1,\dots,1)\}} +
                    \dfrac{\left(1-a^{2^{l-1}}x_1^{2^{l}}\right)}{\left(1-a^{2^{l-1}}y_1^{2^{l}}\right)} \Bc_{\{A^{[l,m]}(2,\dots,2)\}} \nonumber \\
                    &&+\sum\limits_{k=1}^{m-2} \Bc_{\{A^{[k,k+1]}(2,1)\}} + \Bc_{\{ A^{[m-1,m]}(2,1)\}} \nonumber \\
                    &&+\sum\limits_{k=0}^{m-2} \dfrac{(x_1)^{2^{k}} (1-(ax_1)^{2^{k}})}{(y_1)^{2^{k}} (1-(ay_1)^{2^{k}})}\Bc_{\{A^{[l,m]}_{k}  (
                        1,\dots, \underbrace{1}_{k^{th}}, 2,\dots,2)\}} \nonumber \\
                    &&+\dfrac{(x_1)^{2^{m-1}} (1-(ax_1)^{2^{m-1}})}{(y_1)^{2^{m-1}} (1-(ay_1)^{2^{m-1}})}\Bc_{\{ A^{[l,m]}_{m-1}(1,\dots,1,2)  \}}. \nonumber
                    \end{eqnarray}
                    %$  \mbox{$  \Large\chi $ }^2_0 = \frac{1}{2} $
                    %$ \scalebox{1.5}{$\chi$}^2_0 = \tfrac{1}{2}$

                    \begin{eqnarray}
                    z_{m-1}(\w) &=&  \left(\dfrac{x_1}{y_1}\right)^{2^{m-1}} \Bc_{\{ A^{[l,m]}(1,\dots,1) \}} +
                    \left(\dfrac{x_1}{y_1}\right)^{2^{m-1}} \Bc_{\{A^{[l,m]}_{m-1}(1,\dots,1,2)\}} \nonumber \\
                    && +\dfrac{\left(1-a^{2^{l-1}}x_1^{2^{l}}\right)}{\left(1-a^{2^{l-1}}y_1^{2^{l}}\right)} \Bc_{\{A^{[l,m]}  (
                        2,\dots,2, 1)\}} +
                        \dfrac{\left(1-a^{2^{l-1}}x_1^{2^{l}}\right)}{\left(1-a^{2^{l-1}}y_1^{2^{l}}\right)} \Bc_{\{A^{[l,m]}(2,\dots,2)\}} \nonumber \\
                    &&+\sum\limits_{k=1}^{m-2} \Bc_{\{A^{[k,k+1,m]}(2,1,1)\}} + \sum\limits_{k=1}^{m-2} \Bc_{\{A^{[k,k+1,m]}(2,1,2)\}}  \nonumber \\
                    &&+\sum\limits_{k=1}^{m-2} \dfrac{(x_1)^{2^{k}} (1-(ax_1)^{2^{k}})}{(y_1)^{2^{k}} (1-(ay_1)^{2^{k}})}\Bc_{\{A^{[l,m]}_{k}  (
                        1,\dots, \underbrace{1}_{k^{th}}, 2,\dots,2,1)\}}  \nonumber\\
                    &&+\sum\limits_{k=1}^{m-2} \dfrac{(x_1)^{2^{k}} (1-(ax_1)^{2^{k}})}{(y_1)^{2^{k}} (1-(ay_1)^{2^{k}})}\Bc_{\{A^{[l,m]}_{k}  (
                        1,\dots, \underbrace{1}_{k^{th}}, 2,\dots,2,2)\}}. \nonumber
                    \end{eqnarray}
                    Here, we have used a hint $ \Bc_{\{A^{[l,m-1]}(i_{l},\dots,i_{m-1})\}} =  \Bc_{\{A^{[l,m]}(i_{l},\dots,i_{m-1},1)\}}+  \Bc_{\{A^{[l,m]}(i_{l},\dots,i_{m-1},2)\}}$.

                    Now, we are ready to compute $\a_m$.
                    \begin{eqnarray}
                    \a_m(\w) &=& \left(\dfrac{x_1}{y_1}\right)^{2^{m-1}} \Bc_{\{A^{[l,m]}(1,\dots,1)\}} +
                    \dfrac{ 1-(ax_1)^{2^{m-1}}}{ 1-(ay_1)^{2^{m-1}}}\Bc_{\{A^{[l,m]}_{m-1}(1,\dots,1,2)\}} \nonumber \\
                    &&+\sum\limits_{k=1}^{m-2} \Bc_{\{A^{[k,k+1,m]}(2,1,1)\}} + \sum\limits_{k=1}^{m-2} \Bc_{\{A^{[k,k+1,m]}(2,1,2)\}} \nonumber \\
                    &&+\sum\limits_{k=1}^{m-2} \dfrac{(y_1)^{2^{k}} (1-(ay_1)^{2^{k}})}{(x_1)^{2^{k}} (1-(ax_1)^{2^{k}})}\Bc_{\{A^{[l,m]}_{k}  (
                        1,\dots, \underbrace{1}_{k^{th}}, 2,\dots,2,1)\}}  \nonumber \\
                    && +\dfrac{\left(1-a^{2^{l-1}}y_1^{2^{l}}\right)}{\left(1-a^{2^{l-1}}x_1^{2^{l}}\right)} \Bc_{\{A^{[l,m]}  (
                        2,\dots,2, 1)\}} + \Bc_{\{A^{[l,m]}  (
                        2,\dots,2,)\}}  \nonumber \\
                    && + \sum\limits_{k=1}^{m-2} \Bc_{\{A^{[l,m]}_{k}  (
                        1,\dots, \underbrace{1}_{k^{th}}, 2,\dots,2)\}}. \nonumber
                    \end{eqnarray}
                    We stress that, $ \a_m $ is bounded almost everywhere, except for the case $ \w = (\w_0 = 1, \w_1 = 1, \dots, ) $ and $ x_1 > y_1 $. Indeed, the measure $ \m_{\yb, V_{a}} $ at this point is zero. Therefore, one finds that $ \a_m $ satisfies the condition \eqref{thm_eqn_con_test} (see Theorem \ref{thm_test_re_m}). From
                    \begin{eqnarray}
                    1 &=&  \Bc_{\{A^{[l,m]}(1,\dots,1)\}} +  \Bc_{\{A^{[l,m]}_{m-1}(1,\dots,1,2)\}} +\sum\limits_{k=1}^{m-2} \Bc_{\{A^{[k,k+1,m]}(2,1,1)\}}
                    \nonumber \\ &&+
                    \sum\limits_{k=1}^{m-2} \Bc_{\{A^{[k,k+1,m]}(2,1,2)\}} + \sum\limits_{k=1}^{m-2} \Bc_{\{A^{[l,m]}_{k}  (
                        1,\dots, \underbrace{1}_{k^{th}}, 2,\dots,2,1)\}}  \nonumber \\
                    && + \Bc_{\{A^{[l,m]}  (
                        2,\dots,2, 1)\}} + \Bc_{\{A^{[l,m]}  (
                        2,\dots,2,)\}} + \nonumber   \sum\limits_{k=1}^{m-2} \Bc_{\{A^{[l,m]}_{k}  (
                        1,\dots, \underbrace{1}_{k^{th}}, 2,\dots,2)\}}. \nonumber
                    \end{eqnarray}
                    we find
                    \begin{eqnarray} \label{eqn_form_1-alp_sq_V_{a}}
                    (1-\a_m)^2
                    &=&\sum\limits_{k=1}^{m-2} \left( 1-  \dfrac{(y_1)^{2^{k}} (1-(ay_1)^{2^{k}})}{(x_1)^{2^{k}} (1-(ax_1)^{2^{k}})} \right)^{2}\Bc_{\{A^{[l,m]}_{k}  (
                        1,\dots, \underbrace{1}_{k^{th}}, 2,\dots,2,1)\}}  \nonumber \\
                    &&+ \left( 1- \dfrac{ 1-(ax_1)^{2^{m-1}}}{ 1-(ay_1)^{2^{m-1}}}  \right)^{2} \Bc_{\{A^{[l,m]}_{m-1}(1,\dots,1,2)\}} \nonumber \\
                    &&+\left( 1- \dfrac{1-y_1}{1-x_1} \right)^{2} \Bc_{\{A^{[l,m]}  (
                        2,\dots,2, 1)\}}  \\
                    &&+\left(  1-\left(\dfrac{x_1}{y_1}\right)^{2^{m-1}} \right) ^{2} \Bc_{\{A^{[l,m]}(1,\dots,1)\}}. \nonumber
                    \end{eqnarray}
                    %Next, we need to compute the expectation value for $(1-\a_m)^{2}$ given that $\gf_{m-1}$ (i.e $E((1-\a_m)^{2}| \gf_{m-1})$). Let us examine how we want to calculate this expectation.
                    %Denote sets
                    %\begin{eqnarray}
                    %r^{(m)} &= & (1-\a_m)^{2} \nonumber \\
                    %r_k & = & \left( 1- \sum\limits_{k=0}^{m-2} \dfrac{(y_1)^{2^{k}} (1-(ay_1)^{2^{k}})}{(x_1)^{2^{k}} (1-(ax_1)^{2^{k}})} \right)^{2} \nonumber \\
                    %r_{m-1} & = & \left( 1- \dfrac{1-(ax_1)^{2^{m-1}}}{1-(ay_1)^{2^{m-1}}}  \right)^{2} \nonumber \\
                    %r_{m} & = & \left( 1- \dfrac{1-ay_1^2}{1-ax_1^2} \right)^{2} \nonumber \\
                    %r_{m+1} & = & \left(  1-\left(\dfrac{x_1}{y_1}\right)^{2^{m-1}} \right) ^{2} \nonumber
                    %\end{eqnarray}
                    %and
                    %\begin{eqnarray}\label{eqn_ntn_set}
                    %A_{0,k}^{(m)} & = & \{i_{0} = \dots = i_k = 1, i_{k+1} = \dots = i_{m-1} = 2, i_m=1 \} \nonumber \\
                    %A_{0,m-1}^{(m)} & = & \{ i_{0} = \dots = i_{m-2} = i_{m-1} = 1, i_{m} = 2 \} \nonumber\\
                    %A_{0,m}^{(m)} & = & \{ i_{0}= \dots =i_{m-1} = 2, i_m=1 \} \nonumber\\
                    %A_{0,m+1}^{(m)} & = & \{ i_{0}=\dots =i_{m-1} = 1, i_m=1 \}  \\
                    %\th_{0,k}^{[]} & = & \{i_{0} = \dots = i_k = 1, i_{k+1} = \dots = i_{m-1} = 2 \} \nonumber \\
                    %\th_{0,m-1}^{(m-1)} & = & \{ i_{0} = \dots = i_{m-2} = i_{m-1} = 1 \} \\
                    %\th_{0,m}^{(m-1)} & = & \{ i_{0}= \dots =i_{m-1} = 2 \} \nonumber
                    %\ th^{(m)} & = & \bigcup\limits_{i=1}^{m}\th_i \nonumber
                    %\end{eqnarray}
                    %where $ m,k \in \bn, m\geq 1  $ and $k = \{ 1, 2, \dots, m-2\}$ is defined when $ m > 3 $.\\
                    The crucial point in this proof is to make proper partitions on $\w$ with regard to the convergence of the following series:
                    \begin{eqnarray}\label{eqn_seri}
                    \sum\limits_{m=1}^{\infty} E((1-\a_m)^{2}|\gf_{m-1}) (\w).
                    \end{eqnarray}

                    From \eqref{eqn_form_1-alp_sq_V_{a}}, one concludes that it is enough to compute the expectation on the events
                    \[ A^{[l,m-1]}_{k}(1,\dots, \underbrace{1}_{k^{th}}, 2, \dots, 2), \ A^{[l,m-1]}(1,\dots,1) \textmd{ and } A^{[l,m-1]}(2,\dots,2) \]
                    %It is clear that, the space $ \gf_{m-1} $ can be reduced to $ A^{[l,m-1]}_{k}(1,\dots, \underbrace{1}_{k^{th}}, 2, \dots, 2)$, $A^{[l,m-1]}(1,\dots,1) $ and $ A^{[l,m-1]}(2,\dots,2) $ only (see $ (1-\a_m)^{2} $).
                    By virtue of Remark \ref{rem_exptn}, we have
                    \[
                    \begin{array}{l}
                    \m_{\xb,V_{a}} \left( A^{[l,m]}_{k}(1,\dots, \underbrace{1}_{k^{th}}, 2, \dots, 2,1) \left\vert A^{[l,m-1]}_{k}(1,\dots, \underbrace{1}_{k^{th}}, 2, \dots, 2) \right.\right) = \\
                    = \dfrac{ \m_{\xb,V_{a}} \left( A^{[l,m]}_{k}(1,\dots, \underbrace{1}_{k^{th}}, 2, \dots, 2,1) \right) }{\m_{\xb,V_{a}} \left( A^{[l,m-1]}_{k}(1,\dots, \underbrace{1}_{k^{th}}, 2, \dots, 2) \right) } =0 \nonumber
                    \end{array}
                    \]
                    for any $ k\in\{ l,\dots, m-2 \} $. We note
                    that
                    \begin{eqnarray}
                    \m_{\xb,V_{a}} \left( A^{[l,m]}(1,\dots,1) \left\vert A^{[l,m-1]}(1,\dots,1) \right. \right)  &=& (ax_1)^{2^{m-1}} \nonumber \\
                    \m_{\xb,V_{a}} \left( A^{[l,m]}(2,\dots,2,1) \left\vert A^{[l,m-1]}(2,\dots,2) \right. \right)  &=&  0 \nonumber \\
                    \m_{\xb,V_{a}} \left( A^{[l,m]}_{m-1}(1,\dots,1,2) \left\vert A^{[l,m-1]}(1,\dots,1) \right. \right) & = & 1 - (ax_1)^{2^{m-1}} \nonumber
                    \end{eqnarray}
                    For other possibilities, the measures equal to zero, since the intersection between the sets are empty set.

                    Hence, the conditional expectation of $(1-\a_m)^{2}$ can be calculated as follows
                    \begin{eqnarray}\nonumber
                    E((1-\a_m)^{2}|\gf_{m-1}) & = & \left( 1- \dfrac{1-(ax_1)^{2^{m-1}}}{1-(ay_1)^{2^{m-1}}}  \right)^{2} \left( 1- (ax_1)^{2^{m-1}} \right)
                    \Bc_{\{A^{[l,m]}_{m-1}(1,\dots,1,2)\}} \nonumber \\
                    &&+ \left(1-\left(\dfrac{x_1}{y_1}\right)^{2^{m-1}}\right)^{2} \left( (ax_1)^{2^{m-1}} \right) \Bc_{\{A^{[l,m]}(1,\dots,1)\}}. \nonumber
                    \end{eqnarray}

                    Now we consider several cases.\\
                    \textbf{First Case}.
                    Consider the series
                    \begin{eqnarray}\label{eqn_ser}
                    \sum\limits_{m=1}^{\infty} \left( 1- \dfrac{1-(ax_1)^{2^{m-1}}}{1-(ay_1)^{2^{m-1}}}  \right)^{2} \left( 1- (ax_1)^{2^{m-1}} \right) \Bc_{\{ A^{[l,m]}_{m-1}(1,\dots,1,2) \}} (\w).
                    \end{eqnarray}
For the sake of simplicity, put
                    \[ K_m = \left(\dfrac{1-(ax_1)^{2^{m-1}}}{1-(ay_1)^{2^{m-1}}}  \right)^{2} \left( 1- (ax_1)^{2^{m-1}} \right). \]
To study the series \eqref{eqn_ser}, we examine two subcases.
First, we consider elements $\w$ in the following form:
                    \begin{eqnarray}\label{eqn_w_1,...,1,2,..}
                    \w = \{\underbrace{1,\dots, 1, 2}_{N+1}, i_{N+1}, i_{N+2}, \dots\}.
                    \end{eqnarray}
Then substituting $\w$ into \eqref{eqn_ser} one finds
                    \begin{eqnarray}
                    %\sum\limits_{m=1}^{\infty} E((1-\a_m)^{2} | \gf_{m-1}) & = & \sum\limits_{m=1}^{\infty} \left( 1- \dfrac{(x_1)^{-2} (1-(ax_1)^{2^{m}})}{(y_1)^{-2} (1-(ay_1)^{2^{m}})}  \right)^{2} \left( 1- (ax_1)^{2^{m}} \right) \l_{\{F_{m}\}} (\w) \nonumber \\
                    \sum\limits_{m=1}^{\infty}K_m  \Bc_{\{ A^{[0,m]}_{m-1}(1,\dots,1,2) \}}(\w) & = & \left( 1- \dfrac{ 1-(ax_1)^{2^{N-1}}}{1-(ay_1)^{2^{N-1}}}  \right)^{2} \left( 1- (ax_1)^{2^{N-1}} \right) \Bc_{\{A^{[0,N]}_{N-1}(1,\dots,1,2)\}}.\nonumber
                    \end{eqnarray}

If $\w$ is not in the form \eqref{eqn_w_1,...,1,2,..}),  then $
\sum_{m=1}^{\infty}K_m
\Bc_{\{A^{[0,m]}_{m-1}(1,\dots,1,2)\}}(\w)=0. $

                    These two subcases give us
                    \begin{eqnarray}
                    \m_{\xb,V_{a}} \left\{ \w \in \O \left\vert \sum\limits_{m=1}^{\infty}  \right. K_m  \Bc_{\{A^{[0,m]}(1,\dots,1,2)\}}(\w) < \infty \right\} = 1.\nonumber
                    \end{eqnarray}

\textbf{Second Case}. Now we examine the following series:
                    \begin{eqnarray}\label{eqn_ser_2}
                    \sum\limits_{m=1}^{\infty}\left(1-\left(\dfrac{x_1}{y_1}\right)^{2^{m-1}}\right)^{2} \left( (ax_1)^{2^{m-1}} \right)  \Bc_{\{A^{[0,m]}(1,\dots,1)\}}(\w),
                    \end{eqnarray}
                    Denote
                    \[ \hat{K}_m = \left(1-\left(\dfrac{x_1}{y_1}\right)^{2^{m-1}}\right)^{2} \left( (ax_1)^{2^{m-1}} \right).  \]

                    Taking $\w$ in the form of \eqref{eqn_w_1,...,1,2,..}, we find
                    %\[ \sum\limits_{m=1}^{\infty} \hat{K}_m \Bc_{\{F_{m}\}} =
                    %\sum\limits_{m=1}^{N-1}\left(1-\left(\dfrac{x_1}{y_1}\right)^{2^{m-1}}\right)^{2}  \left( (ax_1)^{2^{m-1}} \right) \Bc_{\{F_{m}\}}, \]
                    \[ \sum\limits_{m=1}^{\infty} \hat{K}_m \Bc_{\{A^{[0,m]}(1,\dots,1)\}} =
                    \sum\limits_{m=1}^{N-1} \hat{K}_m \Bc_{\{A^{[0,m]}(1,\dots,1)\}}, \]
                    which is finite. Otherwise, if $\w =(1,1,\dots, 1,\dots)$, then the series given by \eqref{eqn_ser_2} is diverged. In other cases, the series will be zero.
                    We can conclude that, the series  is converged almost everywhere except when $\w =(1,1,\dots, 1,\dots)$.

                    From the considered two cases above, we conclude that the series \eqref{eqn_seri} is converged almost everywhere.
                    Hence,
                    \begin{eqnarray}\label{eqn_series_m_x}
                    \m_{\xb,V_{a}} \left( \w \in \O \left\vert \sum\limits_{m=1}^{\infty} \right. E((1-\a_m)^{2} | \gf_{m-1}) (\w) < \infty  \right) = 1.
                    \end{eqnarray}
                    By repeating similar calculations by letting
                    $ z_{m}(\w) = \dfrac{ \m_{\yb,V_{a}} (\w) }{\m_{\xb,V_{a}}(\w) }  $
                    %$ z_m(i_{0},i_1,\dots, i_m) = \frac{ \m_{\yb,V_{a}}(i_0,i_1,\dots, i_m)}{ \m_{\xb,V_{a}}(i_0,i_1,\dots, i_m)}$,
                    one gets
                    \begin{eqnarray}\label{eqn_series_m_y}
                    \m_{\yb,V_{a}} \left( \w \in \O \left\vert \sum\limits_{m=1}^{\infty} \right. \left. E((1-\a_m)^{2} \right\vert \gf_{m-1}) (\w) < \infty \right) = 1.
                    \end{eqnarray}
                    Consequently, from \eqref{eqn_series_m_x} and \eqref{eqn_series_m_y}, we infer
                    that $ \m_{\xb,V_{a}} \sim  \m_{\yb,V_{a}}$.
                    This completes the proof.
                     \end{proof}

From this theorem, we are in a position to formulate the following
conjecture.

\begin{conj} Let $V$ be a regular q.s.o. Then for any $\xb,\yb\in
S^{n-1}$ the corresponding  Markov measures $\m_{\xb,V}$,
$\m_{\yb,V}$ are equivalent, i.e. $\m_{\xb,V} \sim \m_{\yb,V}$.

\end{conj}

%
%
%                \section{Acknowledgments}
%                The present study was conducted  with
%                the supports from the grants  FRGS14-135-0376 of Malaysian Ministry
%                of Education.

\end{document}